\documentclass[12pt]{amsart}
\usepackage{amssymb}
\usepackage{amssymb,amsfonts,amsmath,amsthm,cite,verbatim}
\usepackage{xcolor}
\usepackage{tikz}
\usepackage{graphicx}
\usepackage[left=2.5cm, right=2.5cm, top=3.5cm, bottom=3.5cm]{geometry}

\usepackage{hyperref}

\usepackage[normalem]{ulem}
\newtheorem{theorem}{Theorem}[section]
\newtheorem{thm}[theorem]{Theorem}
\newtheorem{prop}[theorem]{Proposition}
\newtheorem{lem}[theorem]{Lemma}

\newtheorem{cor}[theorem]{Corollary}

\makeatletter \@addtoreset{equation}{section}

\newcommand{\qbinom}[2]{\genfrac{[}{]}{0pt}{}{#1}{#2}}

\DeclareMathOperator*{\CT}{CT}

\begin{document}

\title[]{A $q$-Morris constant term identity for the Lie algebra $A_n$ and its symmetric function generalizations}

\author{Yue Zhou}

\address{School of Mathematics and Statistics, HNP-LAMA, Central South University,
Changsha 410083, P.R. China}

\email{zhouyue@csu.edu.cn}

\subjclass[2010]{05A30, 33D70, 05E05}

\date{March 12, 2023}

\begin{abstract}
\noindent
It is well-known that the Selberg integral is equivalent to the Morris constant term identity.
In 2009 Warnaar obtained the Selberg integral for the Lie algebra $A_n$.
In this paper, from the point view of constant term identities, we obtain a $q$-Morris constant term identity of type $A_n$ and its several symmetric function generalizations. The type $A_n$ $q$-Morris identity looks as if a constant term version of Warnaar's $A_n$ Selberg integral to some extend.

\noindent
\textbf{Keywords:} constant term identity; $q$-Morris identity; symmetric function; Selberg integral;
Warnaar's $A_n$ Selberg integral.
\end{abstract}
\maketitle

\section{Introduction}\label{s-intr}

In 1944, Atle Selberg \cite{Selberg} gave the following remarkable multiple integral:
\begin{multline}\label{e-Selberg}
\int_{0}^1\cdots \int_{0}^1\prod_{i=1}^kz_{i}^{\alpha-1}(1-z_i)^{\beta-1}\prod_{1\leq i<j\leq k}
|z_i-z_j|^{2\gamma}\mathrm{d}z_{1}\cdots \mathrm{d}z_{k} \\
=\prod_{j=0}^{k-1}\frac{\Gamma(\alpha+j\gamma)\Gamma(\beta+j\gamma)
\Gamma\big(1+(j+1)\gamma\big)}{\Gamma\big(\alpha+\beta+(k+j-1)\gamma\big)\Gamma(1+\gamma)},
\end{multline}
where $\alpha, \beta, \gamma$ are complex parameters such that
\[
\mathrm{Re}(\alpha)>0, \quad \mathrm{Re}(\beta)>0,\quad
\mathrm{Re}(\gamma)>-\min\{1/n,\mathrm{Re}(\alpha)/(n-1),\mathrm{Re}(\beta)/(n-1)\}.
\]
When $k=1$, the above Selberg integral reduces to the Euler beta integral.

It is well-known that the Selberg integral is equivalent to the Morris constant term identity \cite{Morris1982}
\begin{equation}\label{Morris}
\CT_{z_0,\dots,z_k} \prod_{i=1}^{k}(1-z_0/z_i)^a(1-z_i/z_0)^b
\prod_{1\leq i\neq j\leq k}(1-z_i/z_j)^c
=\prod_{i=0}^{k-1}\frac{(a+b+ic)!\big((i+1)c\big)!}{(a+ic)!(b+ic)!c!}
\end{equation}
for nonnegative integers $a,b,c$,
where $\CT\limits_{z_i} L(z_i)$ denotes taking the constant term of the Laurent polynomial (series) $L(z_i)$.
Note that we can set $z_0=1$ in \eqref{Morris} since we take the constant term of a homogeneous Laurent polynomial.
In his Ph.D. thesis \cite{Morris1982}, Morris also conjectured the following
$q$-analogue constant term identity
\begin{equation}\label{q-Morris}
\CT_{z_0,\dots,z_k} \prod_{i=1}^{k}(z_0/z_i)_a(qz_i/z_0)_b
\prod_{1\leq i<j\leq k}(z_i/z_j)_c(qz_j/z_i)_c
=\prod_{i=0}^{k-1}\frac{(q)_{a+b+ic}(q)_{(i+1)c}}{(q)_{a+ic}(q)_{b+ic}(q)_{c}},
\end{equation}
where $(y)_c=(y;q)_c:=(1-y)(1-yq)\cdots (1-yq^{c-1})$ is the $q$-factorial for a positive integer $c$ and $(y)_0:=1$.
In 1988,  Habsieger \cite{Habsieger} and Kadell \cite{Kad} independently proved Askey's conjectured
$q$-analogue of the Selberg integral \cite{Askey}.
Expressing their $q$-analogue integral as a constant term identity
they thus proved Morris' $q$-constant term conjecture \eqref{q-Morris}.
Now the constant term identity \eqref{q-Morris} is called the Habsieger--Kadell $q$-Morris identity.

For nonnegative integers $k_1,k_2,\dots,k_n$, we write
\[
z^{(s)}=(z_1^{(s)},\dots,z_{k_s}^{(s)})
\]
and $z=(z^{(1)},\dots,z^{(n)})$. We also make use of the Vandermonde-type products
for alphabets $x=(x_1,\dots,x_l)$ and $y=(y_1,\dots,y_m)$, which are given by
\[
\Delta(x)=\prod_{1\leq i<j\leq l}(x_i-x_j),\qquad \Delta(x,y)=\prod_{i=1}^l\prod_{j=1}^m(x_i-y_j).
\]
In 2009, based on \cite{TV}, Warnaar \cite{War09} obtained a Selberg integral for the Lie algebra $A_n$, see Theorem~\ref{thm-Warnaar} below. Then the original Selberg integral \eqref{e-Selberg} corresponds to the $A_1$ case, and Tarasov and Varchenko's result \cite{TV} is the $A_2$ case.
\begin{thm}\label{thm-Warnaar}
For a positive integer $n$ let $k_1\geq k_2\geq \dots\geq k_n\geq 0$
be integers and $k_0=k_{n+1}=0$. Let $\alpha,\beta_1,\dots,\beta_n,\gamma\in \mathbb{C}$
be such that
\[
\mathrm{Re}(\alpha)>0,\quad \mathrm{Re}(\beta_1)>0,\dots,\mathrm{Re}(\beta_n)>0,
\quad -\min\Big\{\frac{\mathrm{Re}(\alpha)}{k_1-1},\frac{1}{k_1}\Big\}<\mathrm{Re}(\gamma)<\frac{1}{k_1}
\]
and
\[
-\frac{\mathrm{Re}(\beta_s)}{k_s-k_{s+1}-1}<\mathrm{Re}(\gamma)<\frac{\beta_r+\dots+\beta_{s}}{s-r}
\]
for $1\leq r\leq s\leq n$.
Then
\begin{align}\label{e-Warnaar}
&\int_{C_{\gamma}^{k_n,\dots,k_1}[0,1]}\prod_{s=1}^n\Big(|\Delta(z^{(s)})|^{2\gamma}
\prod_{i=1}^{k_s}(z_i^{(s)})^{\alpha_s-1}(1-z_i^{(s)})^{\beta_s-1}\Big)
\prod_{s=1}^{n-1}|\Delta(z^{(s+1)},z^{(s)})|^{-\gamma}\mathrm{d}z\\
&\qquad = \prod_{1\leq r\leq s\leq n}\prod_{i=1}^{k_s-k_{s+1}}
\frac{\Gamma\big(\beta_r+\dots+\beta_s+(i+r-s-1)\gamma\big)}
{\Gamma\big(\alpha_r+\beta_r+\dots+\beta_s+(i+r-s+k_r-k_{r-1}-2)\gamma\big)}\nonumber \\
&\qquad \quad \times
\prod_{s=1}^n\prod_{i=1}^{k_s}\frac{\Gamma\big(\alpha_s+(i-k_{s-1}-1)\gamma\big)\Gamma(i\gamma)}
{\Gamma(\gamma)},\nonumber
\end{align}
where $\alpha_1=\alpha$, $\alpha_2=\dots=\alpha_{n}=1$ and $\mathrm{d}z=\mathrm{d}z^{(1)}\dots \mathrm{d}z^{(n)}$.
\end{thm}
We omit the explicit expression for the integration domain $C_{\gamma}^{k_n,\dots,k_1}[0,1]$ due to its complexity. The reader who is interested can refer \cite[(4.9)]{War09}. Recently, Albion, Rains and Warnaar \cite{ARW} generalized the above integral to an AFLT type by adding a product of two Jack polynomials.

For nonnegative integers $a,b,k_1,k_2$ and a positive integer $c$, let
\begin{equation}\label{defi-L}
L_{k_1,k_2}(a,b,c;z^{(1)},z^{(2)})=
\frac{\prod_{i=1}^{k_1}(z_0/z_i^{(1)})_a(qz_i^{(1)}/z_0)_b
\prod_{1\leq i<j\leq k_1}(z_i^{(1)}/z_j^{(1)})_c(qz_j^{(1)}/z_i^{(1)})_c}
{\prod_{j=1}^{k_2}\prod_{i=1}^{k_1}(z_i^{(1)}/z_j^{(2)})_c}.
\end{equation}
Denote
\begin{equation}\label{defi-L1}
L_{k_1,k_2}(a,b,c):=\CT_{z^{(1)}}L_{k_1,k_2}(a,b,c;z^{(1)},z^{(2)}),
\end{equation}
where we suppress the variables $z^{(2)}$ in the left-hand side.
When taking constant terms, we always explain $(1-dz_i^{(s)}/z_j^{(s+1)})^{-1}$ as
\[
\sum_{i=0}^{\infty}(dz_i^{(s)}/z_j^{(s+1)})^i
\]
for $d\in \mathbb{C}(q)$ and $s=1,\dots,n$ throughout this paper.
That is, we assume all terms of the form $|dz_i^{(s)}/z_j^{(s+1)}|<1$.
In \eqref{defi-L1}, we need the explanation of $(1-dz_i^{(1)}/z_j^{(2)})^{-1}$ to obtain
the constant term.
By convention, we define the product of zero term as 1. Then
\begin{equation*}
L_{k_1,0}(a,b,c)=\CT_{z^{(1)}}\prod_{i=1}^{k_1}(z_0/z_i^{(1)})_a(qz_i^{(1)}/z_0)_b
\prod_{1\leq i<j\leq k_1}(z_i^{(1)}/z_j^{(1)})_c(qz_j^{(1)}/z_i^{(1)})_c
\end{equation*}
equals
\begin{equation}\label{defi-M}
M_{k_1}(a,b,c)=\prod_{i=0}^{k_1-1}\frac{(q)_{a+b+ic}(q)_{(i+1)c}}{(q)_{a+ic}(q)_{b+ic}(q)_{c}}
\end{equation}
by the $q$-Morris identity \eqref{q-Morris}.

For $n$ a positive integer,
let $k_1,\dots,k_{n+1}$, $a_1,\dots,a_n$ and $b_1,\dots,b_n$ be nonnegative integers such that $k_1\geq k_2\geq \cdots \geq k_{n+1}$, $a_1=a,a_2=\cdots=a_n=0$,
and denote $\sigma_s=b_1+\cdots+b_s$ for $s=1,\dots,n$ and $\sigma_0=0$.
In this paper, we find that
\begin{thm}\label{thm-0}
If $a+\sigma_s+s\geq sc$ for $s=1,\dots,n$ and $k_{n+1}=0$, then
\begin{multline}\label{e-MorrisA}
\CT_{z}
\prod_{s=1}^{n}\frac{\prod_{i=1}^{k_s}(1/z_i^{(s)})_{a_s}(qz_i^{(s)})_{b_s}
\prod_{1\leq i<j\leq k_s}(z_i^{(s)}/z_j^{(s)})_c(qz_j^{(s)}/z_i^{(s)})_c}
{\prod_{i=1}^{k_s}(z_i^{(s)})^{b_s+1-c}\prod_{j=1}^{k_{s+1}}\prod_{i=1}^{k_s}(q^{b_s+1-c}z_i^{(s)}/z_j^{(s+1)})_c}\\
\quad=(-1)^{\sum_{s=1}^nk_s(b_s+1-c)}q^{\sum_{s=1}^nk_s\binom{b_s+2-c}{2}}
 \prod_{s=1}^nM_{k_s}\big(a+\sigma_s+s(1-c),c-1,c\big).
\end{multline}
\end{thm}
If readers compare the similarity of \eqref{e-MorrisA} and \eqref{e-Warnaar}
with the equivalence between the Morris identity \eqref{Morris} and the Selberg integral \eqref{e-Selberg},
 they may find that the constant term identity \eqref{e-MorrisA} looks as if a constant term version of Warnaar'a $A_n$ Selberg integral \eqref{e-Warnaar}.
But we do not know how to connect them so far.
Note that in this paper, we obtain more general results than Theorem~\ref{thm-0}
by adding symmetric functions, see Theorem~\ref{thm-main-1} and Theorem~\ref{thm-main-2} below.

The structure of this paper is as follows. In the next section, we introduce some basic notation and
results of symmetric functions and plethystic notation. In Section~\ref{s-split}, we introduce the main tool of this paper --- a splitting formula. Using the formula, we find a family of vanishing constant terms.
In Section~\ref{sec-expression},  we determine the expression of $L_{k_1,k_2}(a,b,c)$.
In Section~\ref{sec-bc1}, we study the $b=c-1$ case of $L_{k_1,k_2}(a,b,c)$.
In the last section, we obtain several results equivalent to those in Section~\ref{sec-bc1}.

\section{Symmetric functions and plethystic notation}

In this section, we introduce some basic results of symmetric functions and plethystic notation.

A partition is a sequence $\lambda={(\lambda_1,\lambda_2,\dots)}$ of nonnegative integers such that
${\lambda_1\geq \lambda_2\geq \cdots}$ and
only finitely many $\lambda_i$ are positive.
The length of a partition $\lambda$, denoted
$\ell(\lambda)$ is defined to be the number of nonzero $\lambda_i$.
We adopt the convention of not displaying the tail of zeros of a partition.
We say that $|\lambda|=\lambda_1+\lambda_2+\cdots$ is the size of the partition $\lambda$.
The most commonly used partial order for partitions should be the dominance order, denoted $\leq$.
If $\lambda$ and $\mu$ are partitions of the same size, then $\lambda\leq \mu$ if
$\lambda_1+\cdots+\lambda_i\leq \mu_1+\cdots+\mu_i$ for all $i\geq 1$.

For a field $\mathbb{F}$, denote $\Lambda_{\mathbb{F}}$ by the ring of symmetric functions with coefficients in $\mathbb{F}$.
In the following, we introduce four bases of $\Lambda$ (with $\mathbb{F}=\mathbb{Q}$ and $\mathbb{F}=\mathbb{Q}(q,t)$).

Let $X=\{x_1,x_2,\dots\}$ be an alphabet (a set of countably many variables).
For $r$ a positive integer, let $p_r$ be the power sum
symmetric function in $X$, defined by
\[
p_r=\sum_{i\geq 1}x_i^r.
\]
In addition, we set $p_0=1$.
For a partition $\lambda=(\lambda_1,\lambda_2,\dots)$, let
\[
p_{\lambda}=p_{\lambda_1}p_{\lambda_2}\cdots.
\]
The $p_r$ are algebraically independent over $\mathbb{Q}$,
and the $p_{\lambda}$ form a basis of $\Lambda_{\mathbb{Q}}$ \cite{Mac95}. That is,
\[
\Lambda_{\mathbb{Q}}=\mathbb{Q}[p_1,p_2,\dots].
\]

The elementary symmetric function is defined by
\[
e_r=\sum_{1\leq i_1<\dots<i_r}x_{i_1}\cdots x_{i_r}, \quad e_0=1,
\quad \mbox{and} \quad
e_{\lambda}=e_{\lambda_1}e_{\lambda_2}\cdots.
\]
The $e_{\lambda}$ form a basis of $\Lambda_{\mathbb{Q}}$ \cite{Mac95}.

The complete symmetric function $h_r(X)$
can be defined in terms of its generating function as
\begin{equation}\label{e-gfcomplete}
\sum_{r\geq 0}  h_r(X)y^r=\prod_{i\geq 0}\frac{1}{1-yx_i}.
\end{equation}
Like $p_{\lambda}$ and $e_{\lambda}$, define
\[
h_{\lambda}=h_{\lambda_1}h_{\lambda_2}\cdots.
\]
The $h_{\lambda}$ also form a basis of $\Lambda_{\mathbb{Q}}$ \cite{Mac95}.
(The $e_{\lambda}$ and the $h_{\lambda}$ are in fact bases of $\Lambda_{\mathbb{Z}}$,
but we do not need these in this paper.)

Let $F=\mathbb{Q}(q,t)$ be the field of rational functions in $q$ and $t$ with coefficients in $\mathbb{Q}$ and $(y)_{\infty}=(y;q)_{\infty}:=(1-y)(1-yq)\cdots $ be the infinity $q$-factorial.
Define the modified complete symmetric function $g_r(X;q,t)\in \Lambda_{F}(X)$ by its generating function
\[
\sum_{r\geq 0}g_r(X;q,t)y^r=\prod_{i\geq 1}\frac{(tx_iy;q)_{\infty}}{(x_iy;q)_{\infty}},
\]
and for a partition $\lambda=(\lambda_1,\lambda_2,\dots)$ define
\[
g_{\lambda}=g_{\lambda}(q,t)=g_{\lambda}(X;q,t):=\prod_{i\geq 1}g_{\lambda_i}(X;q,t).
\]
By \cite[Chapter VI, (2.12)]{Mac95}, the $g_{\lambda}$ form a basis of $\Lambda_{F}$.

Plethystic or $\lambda$-ring notation is a device to facilitate computations in the ring of symmetric functions. We briefly introduce plethystic notation. For more details, see \cite{haglund,Lascoux,RW}.

For an alphabet $X=\{x_1,x_2,\dots\}$ and a field $\mathbb{F}$, we additively write $X:=x_1+x_2+\cdots$, and
use plethystic brackets to indicate this additive notation:
\[
f(X)=f(x_1,x_2,\dots)=f[x_1+x_2+\cdots]=f[X], \quad
\text{for $f\in \Lambda_{\mathbb{F}}$.}
\]
We introduce a consistent arithmetic on alphabets in terms of the basis
of power sums. In particular, a power sum whose argument is the sum, difference or Cartesian product
of two alphabets $X$ and $Y$ is defined as
\begin{subequations}\label{asm}
\begin{align}
\label{asm1}
p_r[X+Y]&=p_r[X]+p_r[Y], \\
p_r[X-Y]&=p_r[X]-p_r[Y], \\
p_r[XY]&=p_r[X]p_r[Y].\label{asm3}
\end{align}
\end{subequations}
In general we cannot give meaning to division by an arbitrary
alphabet and only division by $1-t$ (the difference of two one-letter
alphabets with ``letters'' $1$ and $t$ respectively) is meaningful.
In particular
\begin{equation}\label{division}
p_r\Big[\frac{X}{1-t}\Big]=\frac{p_r[X]}{1-t^r}.
\end{equation}
Note that the alphabet $1/(1-t)$ may be interpreted as
the infinite alphabet $1+t+t^2+\cdots$.
Indeed, by \eqref{asm1} and \eqref{asm3}
\[
p_r[X(1+t+t^2+\cdots)]=p_r[X]\sum_{k=0}^{\infty} p_r[t^k]=
p_r[X]\sum_{k=0}^{\infty} t^{kr}=\frac{p_r[X]}{1-t^r}.
\]

If $f$ is a homogeneous symmetric function of degree $k$ then
\begin{equation}\label{e-homo-sym}
f[aX]=a^kf[X]
\end{equation}
for a single-letter alphabet $a$.
Note that by \eqref{asm1}, $f[2X]=f[X+X]=2f[X]$.
Moreover, we can extend this to
\begin{equation}\label{e-homo-2}
f[rX]=rf[X]
\end{equation}
for $r\in \mathbb{F}$.
Note that this leads to some notational ambiguities,
and whenever not clear from the context we will indicate if a symbol such as $a$ or $r$
represents a letter or a binomial element\footnote{In \cite[p. 32]{Lascoux} Lascoux refers to $r\in \mathbb{F}$ as a binomial element.}.
Throughout this paper, we take $q$ and $q^{-1}$ as single-letter alphabets.

We need the next simple result.
One can find a proof in \cite[Theorem 1.27]{haglund}.
\begin{prop}\label{Ple-basic}
Let $X$ and $Y$ be two alphabets. For $r$ a nonnegative integer,
\begin{align}\label{e-xy}
h_{r}[X+Y]&=\sum_{i=0}^rh_i[X]h_{r-i}[Y], \\
h_{r}[-X]&=(-1)^re_r[X].\label{e-he}
\end{align}
\end{prop}

Using plethystic notation, we can write $g_r$ as
\begin{equation}\label{modi-complete}
g_r(X;q,t)=h_r\Big[\frac{1-t}{1-q}X\Big].
\end{equation}

By \eqref{e-xy} and \eqref{modi-complete}, it is easy to obtain the next result.
\begin{prop}
Let $X$ and $Y$ be two alphabets. For $r$ a nonnegative integer,
\begin{equation}\label{e-g}
g_r[X+Y]=\sum_{i=0}^rg_i[X]g_{r-i}[Y].
\end{equation}
\end{prop}
\begin{proof}
By \eqref{modi-complete},
\[
g_r[X+Y]=h_r\Big[\frac{1-t}{1-q}(X+Y)\Big]=
\sum_{i=0}^r h_{r-i}\Big[\frac{1-t}{1-q}X\Big]h_i\Big[\frac{1-t}{1-q}Y\Big].
\]
Here the last equality holds by \eqref{e-xy}.
Using \eqref{modi-complete} again, we obtain \eqref{e-g}.
\end{proof}

\section{A splitting formula}\label{s-split}

In this section, we give a splitting formula for a rational function $F_n(y,w)$ defined in \eqref{FD} below.
Using the splitting formula, we find a family of vanishing constant terms.

Before presenting the splitting formula, we need the next simple result.
\begin{lem}
Let $i$ and $j$ be positive integers. Then, for $t$ an integer such that $0\leq t\leq j$,
\begin{subequations}\label{e-ab}
\begin{equation}\label{prop-b1}
\frac{(1/y)_{i}(qy)_j}{(q^{-t}/y)_i}=q^{it}(q^{1-i}y)_t(q^{t+1}y)_{j-t},
\end{equation}
and for $-1\leq t\leq j-1$,
\begin{equation}\label{prop-b2}
\frac{(y)_j(q/y)_i}{(q^{-t}/y)_i}=q^{i(t+1)}(q^{-i}y)_{t+1}(q^{t+1}y)_{j-t-1}.
\end{equation}
\end{subequations}
\end{lem}
Note that the $t=j$ case of \eqref{prop-b1} (taking $y\mapsto y/q$) is the standard fact in \cite[Equation~(I.13)]{GR}.
\begin{proof}
For $0\leq t\leq j$,
\[
\frac{(1/y)_{i}(qy)_j}{(q^{-t}/y)_i}=\frac{(q^{i-t}/y)_t(qy)_j}{(q^{-t}/y)_t}
=\frac{(-1/y)^tq^{it-\binom{t+1}{2}}(q^{1-i}y)_t(qy)_j}{(-1/y)^{t}q^{-\binom{t+1}{2}}(qy)_t}
=q^{it}(q^{1-i}y)_t(q^{t+1}y)_{j-t}.
\]

Taking $y\mapsto y/q$ and $t\mapsto t+1$ in \eqref{prop-b1} yields
\eqref{prop-b2} for $-1\leq t\leq j-1$.
\end{proof}

For positive integers $n,c$, let
\begin{equation}\label{FD}
F_n(y,w)=\frac{\prod_{1\leq i<j\leq n}(y_i/y_j)_c(qy_j/y_i)_c}{\prod_{l=1}^{n}(y_l/w)_c},
\end{equation}
where $w$ is a parameter such that all terms of the form $q^uy_i/w$ in \eqref{FD}
satisfy $|q^uy_i/w|<1$. Hence,
\[
\frac{1}{1-q^uy_i/w}=\sum_{j\geq 0} (q^uy_i/w)^j.
\]
The rational function $F_n(y,w)$ admits the following partial fraction expansion.
We refer it as the splitting formula for $F_n(y,w)$.
\begin{prop}\label{prop-split}
Let $F_n(y,w)$ be defined as in \eqref{FD}. Then
\begin{equation}\label{e-Fsplit}
F_n(y,w)=\sum_{i=1}^{n}\sum_{j=0}^{c-1}\frac{A_{ij}}{1-q^jy_i/w},
\end{equation}
where
\begin{multline}\label{A}
A_{ij}=\frac{q^{(n-1)jc+(n-i)c}}{(q^{-j})_j(q)_{c-j-1}}
\prod_{l=1}^{i-1}\big(q^{1-c}y_i/y_l\big)_{j}\big(q^{j+1}y_i/y_l\big)_{c-j}
\prod_{l=i+1}^n\big(q^{-c}y_i/y_l\big)_{j+1}\big(q^{j+1}y_i/y_l\big)_{c-j-1} \\
\times \prod_{\substack{1\leq u<v\leq n\\ u,v\neq i}}(y_u/y_v)_c(qy_v/y_u)_c.
\end{multline}
\end{prop}
Note that the $A_{ij}$ are polynomials in $y_i$. Proposition~\ref{prop-split}
is the $\beta_i=c$ case of \cite[Theorem 3.3]{cai} and the $n_0=0$ case of \cite[Proposition 3.5]{GZ}.
For completeness, we give a proof below.
\begin{proof}
By partial fraction decomposition of $F_n(y,w)$ with respect to $1/w$,
we can rewrite $F_n(y,w)$ as \eqref{e-Fsplit} and
\begin{equation}\label{Aij}
A_{ij}=F_n(y,w)(1-q^jy_i/w)|_{w=q^jy_i} \quad \text{for $i=1,\dots,n$ and $j=0,\dots,c-1$.}
\end{equation}

Carrying out the substitution $w=q^jy_i$ in $F_n(y,w)(1-q^jy_i/w)$ for $i=1,\dots,n$ yields
\begin{multline}\label{BB}
A_{ij}=\frac{1}{(q^{-j})_j(q)_{c-j-1}}
\prod_{l=1}^{i-1}\frac{(y_l/y_i)_{c}(qy_i/y_l)_{c}}{(q^{-j}y_l/y_i)_{c}}
\prod_{l=i+1}^n\frac{(y_i/y_l)_{c}(qy_l/y_i)_{c}}{(q^{-j}y_l/y_i)_{c}}\\
\times
\prod_{\substack{1\leq u<v\leq n\\v,u\neq i}}
(y_u/y_v)_c(qy_v/y_u)_c.
\end{multline}
Using \eqref{e-ab} with $(i,j,t,y)\mapsto (c,c,j,y_i/y_l)$,
we have
\begin{subequations}\label{B34}
\begin{equation}\label{B3}
\frac{(y_l/y_i)_{c}(qy_i/y_l)_{c}}{(q^{-j}y_l/y_i)_{c}}
=q^{jc}\big(q^{1-c}y_i/y_l\big)_j\big(q^{j+1}y_i/y_l\big)_{c-j}
\quad \text{for $j=0,\dots,c$},
\end{equation}
and
\begin{equation}\label{B4}
\frac{(y_i/y_l)_{c}(qy_l/y_i)_{c}}{(q^{-j}y_l/y_i)_{c}}
=q^{(j+1)c}\big(q^{-c}y_i/y_l\big)_{j+1}\big(q^{j+1}y_i/y_l\big)_{c-j-1}
\quad \text{for $j=-1,\dots,c-1$}.
\end{equation}
\end{subequations}
Substituting \eqref{B34} into \eqref{BB} gives \eqref{A}.
\end{proof}

By Proposition~\ref{prop-split}, we find that a family of constant terms vanish.
\begin{lem}\label{lem-vanish}
Let $k_1,k_2,c$ be integers such that $c>0$ and $k_1>k_2\geq 0$.
For integers $t_1,t_2,\dots,t_{k_1}$,
set $\{t_l\leq 0\mid l=1,\dots,k_1\}=\{b_1,\dots,b_m\}$
and $\{t_l>0\mid l=1,\dots,k_1\}=\{d_1,\dots,d_p\}$ (assume $d_1\leq d_2\leq \cdots\leq d_p$).
If $\sum_{l=1}^mb_l+\sum_{l=1}^{p-k_2}d_l>0$ then
\[
C:=\CT_{z^{(1)}}\frac{\prod_{1\leq i<j\leq k_1}(z_i^{(1)}/z_j^{(1)})_c(qz_j^{(1)}/z_i^{(1)})_c}
{\prod_{u=1}^{k_1}(z^{(1)}_u)^{t_u}\prod_{i=1}^{k_1}\prod_{j=1}^{k_2}(z_i^{(1)}/z_j^{(2)})_c}=0.
\]
\end{lem}
Note that if $p\leq k_2$ then we set the sum $\sum_{j=1}^{p-k_2}d_j=0$ by convention. The condition
$\sum_{l=1}^mb_l+\sum_{l=1}^{p-k_2}d_l>0$ can not hold for $k_2\geq k_1$. Because in this case
$\sum_{l=1}^{p-k_2}d_l=0$ by $k_2\geq k_1\geq p$ and all the $b_l\leq 0$.
\begin{proof}
For $k_2=0$
\[
C=\CT_{z^{(1)}}\prod_{i=1}^{k_1}(z^{(1)}_i)^{-t_i}
\prod_{1\leq i<j\leq k_1}(z_i^{(1)}/z_j^{(1)})_c(qz_j^{(1)}/z_i^{(1)})_c.
\]
The above constant term vanishes since the Laurent polynomial inside the operator $\CT\limits_{z^{(1)}}$
is not homogeneous by $\sum_{i=1}^mb_i+\sum_{j=1}^{p}d_j=\sum_{i=1}^{k_1}t_i>0$. Hence, we can assume $k_2>0$ in the following of the proof.
We proceed the proof by induction on $k_1$. It is clear that the lemma holds for $k_1=1$
(This forces $k_2=0$ by $k_2<k_1$).
Assume the lemma holds for $k_1\mapsto k_1-1$.

Using Proposition~\ref{prop-split} with $(n,y,w)\mapsto (k_1,z^{(1)},z_1^{(2)})$, we can write
\begin{equation}\label{e-vanish1}
C=\CT_{z^{(1)}}\sum_{i=1}^{k_1}\sum_{j=0}^{c-1}\frac{A_{ij}}
{(1-q^jz^{(1)}_i/z^{(2)}_1)\prod_{u=1}^{k_1}(z^{(1)}_u)^{t_u}
\prod_{u=1}^{k_1}\prod_{v=2}^{k_2}(z_u^{(1)}/z_v^{(2)})_c},
\end{equation}
where
\begin{multline*}
A_{ij}=\frac{q^{(k_1-1)jc+(k_1-i)c}}{(q^{-j})_j(q)_{c-j-1}}
\prod_{l=1}^{i-1}\big(q^{1-c}z^{(1)}_i/z^{(1)}_l\big)_{j}\big(q^{j+1}z^{(1)}_i/z^{(1)}_l\big)_{c-j}\\
\times
\prod_{l=i+1}^{k_1}\big(q^{-c}z^{(1)}_i/z^{(1)}_l\big)_{j+1}\big(q^{j+1}z^{(1)}_i/z^{(1)}_l\big)_{c-j-1}
\prod_{\substack{1\leq u<v\leq k_1\\ u,v\neq i}}(z^{(1)}_u/z^{(1)}_v)_c(qz^{(1)}_v/z^{(1)}_u)_c.
\end{multline*}
Denote the summand in \eqref{e-vanish1} by $Q_{ij}$.
We prove that $\CT\limits_{z^{(1)}}Q_{ij}=0$ for all the $i,j$.
Note that we view $Q_{ij}$ as a Laurent series in $z^{(1)}_i$. That is
\begin{equation*}
Q_{ij}=\frac{A_{ij}}
{\prod_{u=1}^{k_1}(z^{(1)}_u)^{t_u}
\prod_{\substack{u=1\\u\neq i}}^{k_1}\prod_{v=2}^{k_2}(z_u^{(1)}/z_v^{(2)})_c}
\sum_{l\geq 0}(q^jz_i^{(1)}/z_{1}^{(2)})^l
\prod_{v=2}^{k_2}\prod_{m=0}^{c-1}\sum_{l\geq 0}(q^mz_i^{(1)}/z_{v}^{(2)})^l.
\end{equation*}
For a fixed integer $i$, if $t_i<0$ then $Q_{ij}$ is in fact a power series in $z^{(1)}_i$ with no constant term.
It follows that $\CT\limits_{z^{(1)}}Q_{ij}=\CT\limits_{z^{(1)}_i}Q_{ij}=0$.
If $t_i\geq 0$ then by the expression for $A_{ij}$ and taking the constant term of $Q_{ij}$ with respect to $z^{(1)}_i$, we can write $\CT\limits_{z^{(1)}_i}Q_{ij}$ as a finite sum of the form
\[
R:=\frac{\prod_{\substack{1\leq u<v\leq k_1\\u,v\neq i}}(z_u^{(1)}/z_v^{(1)})_c(qz_v^{(1)}/z_u^{(1)})_c}
{r\cdot \prod_{\substack{u=1\\u\neq i}}^{k_1}(z^{(1)}_u)^{t_u+s_u}
\prod_{\substack{u=1\\u\neq i}}^{k_1}\prod_{v=2}^{k_2}(z_u^{(1)}/z_v^{(2)})_c},
\]
where $r\in \mathbb{Q}(q)[z^{(2)}]\setminus \{0\}$ and all the $s_u$ are nonnegative integers.
Take $t'_u=t_u+s_u$ and denote $\{t'_u\leq 0\mid u\in\{1,\dots,k_1\}\setminus \{i\}\}=
\{b'_1,\dots,b'_{m'}\}$
and $\{t'_u>0\mid u\in\{1,\dots,k_1\}\setminus \{i\}\}=\{d'_1,\dots,d'_{p'}\}$
(assume $d'_1\leq \cdots\leq d'_{p'}$). We can see that
$\sum_{l=1}^{m'}b'_l\geq \sum_{l=1}^m b_l$, $p'\geq p-1$ and
$\sum_{l=1}^id'_l\geq \sum_{l=1}^id_l$ for any $1\leq i\leq p-1$.
To see the last inequality holds, we can imagine two positive integer sequences:
$d'_1\leq \cdots\leq d'_{p'}$ is obtained from $d_1\leq \cdots\leq d_{p}$
by deleting one element and adding a nonnegative integer to each element.
Then
\[
\sum_{l=1}^{p'-k_2+1}d'_l+\sum_{l=1}^{m'}b'_l\geq \sum_{l=1}^{p-k_2}d'_l+\sum_{l=1}^{m}b_l
\geq \sum_{l=1}^{p-k_2}d_l+\sum_{l=1}^{m}b_l>0.
\]
By the induction hypothesis we conclude that all the form of $\CT\limits_{z^{(1)}}R=0$.
Then $\CT\limits_{z^{(1)}}Q_{ij}$ also vanishes for $t_i\geq 0$.
Since $C$ is a finite sum of the $\CT\limits_{z^{(1)}}Q_{ij}$,  we have $C=0$.
\end{proof}

By Lemma~\ref{lem-vanish}, we can get the next two corollaries.
\begin{cor}\label{cor-1}
Let $k_1,k_2,a,c$ be nonnegative integers and $t_1,\dots,t_{k_1}\in \mathbb{Z}$ such that
$k_1>k_2$, $c>0$, $t_i\leq a$ for $i=1,\dots,k_1$, and $\sum_{i=1}^{k_1}t_i>k_2a$.
Then
\begin{equation*}
\CT_{z^{(1)}}\frac{\prod_{1\leq i<j\leq k_1}(z_i^{(1)}/z_j^{(1)})_c(qz_j^{(1)}/z_i^{(1)})_c}
{\prod_{i=1}^{k_1}(z^{(1)}_i)^{t_i}\prod_{j=1}^{k_2}\prod_{i=1}^{k_1}(z_i^{(1)}/z_j^{(2)})_c}=0.
\end{equation*}
\end{cor}
\begin{proof}
Let $\{t_i\leq 0\mid i=1,\dots,k_1\}=\{b_1,\dots,b_m\}$ and $\{t_i>0\mid i=1,\dots,k_1\}=\{d_1,\dots,d_p\}$
(assume $d_1\leq d_2\leq \cdots\leq d_p$).
By $t_i\leq a$, we have
\[
\sum_{i=1}^mb_i+\sum_{i=1}^{p-k_2}d_i\geq \sum_{i=1}^mb_i+\sum_{i=1}^{p}d_i-k_2a
=\sum_{i=1}^{k_1}t_i-k_2a>0.
\]
The last inequality holds by the condition $\sum_{i=1}^{k_1}t_i>k_2a$ in the corollary.
Hence, the corollary holds by Lemma~\ref{lem-vanish}.
\end{proof}
\begin{cor}\label{cor-2}
Let $k_1,k_2$ be nonnegative integers such that $k_1>k_2$.
Let $c,t_1,\dots,t_{k_1}$ be positive integers.
Then
\begin{equation*}
\CT_{z^{(1)}}\frac{\prod_{1\leq i<j\leq k_1}(z_i^{(1)}/z_j^{(1)})_c(qz_j^{(1)}/z_i^{(1)})_c}
{\prod_{i=1}^{k_1}(z^{(1)}_i)^{t_i}\prod_{j=1}^{k_2}\prod_{i=1}^{k_1}(z_i^{(1)}/z_j^{(2)})_c}=0.
\end{equation*}
\end{cor}
\begin{proof}
Let $B:=\{t_i\leq 0\mid i=1,\dots,k_1\}=\{b_1,\dots,b_m\}$ and $D:=\{t_i>0\mid i=1,\dots,k_1\}=\{d_1,\dots,d_p\}$.
Then $B=\emptyset$ and $D=\{t_1,t_2,\dots,t_{k_1}\}$.
Together with $k_1>k_2$ and all the $t_i\geq 1$, we have
\[
\sum_{i=1}^mb_i+\sum_{i=1}^{p-k_2}d_i=\sum_{i=1}^{k_1-k_2}t_i>0.
\]
By Lemma~\ref{lem-vanish}, the corollary follows.
\end{proof}

\section{The expression for $L_{k_1,k_2}(a,b,c)$}\label{sec-expression}

Recall the definitions of $L_{k_1,k_2}(a,b,c;z^{(1)},z^{(2)})$ and $L_{k_1,k_2}(a,b,c)$
in \eqref{defi-L} and \eqref{defi-L1} respectively. In $L_{k_1,k_2}(a,b,c;z^{(1)},z^{(2)})$,
by expanding the product $\prod_{i=1}^{k_1}(z_0/z_i^{(1)})_a(qz_i^{(1)}/z_0)_b$ directly,
we find that the degree and the lower degree of $z_0$ should be $k_1a$ at most and $-k_1b$ at least respectively.
But under a certain condition, we find that after taking the constant term of $L_{k_1,k_2}(a,b,c;z^{(1)},z^{(2)})$ with respect to $z^{(1)}$ (not w.r.t. $z_0$), the terms of $z_0^l$ for $l<0$ and $l>k_2a$ vanish. More precisely, we will show that $L_{k_1,k_2}(a,b,c)$ is of the form
\begin{equation}\label{e-main}
M_{k_1}(a,b,c)\prod_{j=1}^{k_2}(q^{c-1-b}z_0/z^{(2)}_j)_{a+b+1-c}
\times (1+C_1z_0+\cdots+C_{k_2(c-1-b)}z_0^{k_2(c-1-b)})
\end{equation}
for $k_1\geq k_2$ and $b+1\leq c\leq a+b+1$ in this section.

\subsection{The negative powers}

In this subsection, we show that all the terms of $L_{k_1,k_2}(a,b,c)$ with negative powers of $z_0$
vanish.

For a Laurent series (polynomial) $f$,
denote by $[z^l]f$ the coefficient of $z^l$ in $f$.
\begin{lem}\label{lem-L1}
Let $L_{k_1,k_2}(a,b,c)$ be defined as in \eqref{defi-L1}.
For a negative integer $l$,
\[
[z_0^{l}]L_{k_1,k_2}(a,b,c)=0.
\]
\end{lem}
\begin{proof}
Let
\begin{equation}\label{defi-F}
F_{k_1,k_2}(c):=\frac{\prod_{1\leq i<j\leq k_1}(z_i^{(1)}/z_j^{(1)})_c(qz_j^{(1)}/z_i^{(1)})_c}
{\prod_{j=1}^{k_2}\prod_{i=1}^{k_1}(z_i^{(1)}/z_j^{(2)})_c}.
\end{equation}
For an integer $l$ (not only negative),
\begin{align}
[z_0^{l}]L_{k_1,k_2}(a,b,c)&=\CT_{z_0,z^{(1)}}z_0^{-l}L_{k_1,k_2}(a,b,c)\nonumber \\
&=\CT_{z_0,z^{(1)}}z_0^{-l}\prod_{i=1}^{k_1}(z_0/z_i^{(1)})_a(qz_i^{(1)}/z_0)_b
\times F_{k_1,k_2}(c)\nonumber \\
&=\CT_{z_0,z^{(1)}}z_0^{-l}\prod_{i=1}^{k_1}(-z_0/z_i^{(1)})^aq^{\binom{a}{2}}(q^{1-a}z_i^{(1)}/z_0)_{a+b}
\times F_{k_1,k_2}(c).\label{NP-1}
\end{align}
By the well-known $q$-binomial theorem \cite[Theorem 3.3]{andrew-qbinomial}
\[
(y)_n=\sum_{t\geq 0}q^{\binom{t}{2}}\qbinom{n}{t}(-y)^t,
\]
we can expand \eqref{NP-1} as
\begin{align*}
&[z_0^l]L_{k_1,k_2}(a,b,c)\\
&=\CT_{z_0,z^{(1)}}z_0^{-l}\prod_{i=1}^{k_1}(-z_0/z_i^{(1)})^aq^{\binom{a}{2}}
\sum_{t_i\geq 0}q^{\binom{t_i}{2}}\qbinom{a+b}{t_i}(-q^{1-a}z_i^{(1)}/z_0)^{t_i}\times F_{k_1,k_2}(c)\\
&=\CT_{z_0,z^{(1)}}\sum_{t_1,\dots,t_{k_1}\geq 0}
z_0^{k_1a-|t|-l}\prod_{i=1}^{k_1}(-z^{(1)}_i)^{t_i-a}
q^{\binom{a}{2}+\binom{t_i}{2}+(1-a)t_i}\qbinom{a+b}{t_i}\times F_{k_1,k_2}(c).
\end{align*}
Here $|t|:=\sum_{i=1}^{k_1}t_i$ and $\qbinom{n}{t}=(q^{n-t+1})_t/(q)_t$ is the $q$-binomial coefficient
for nonnegative integers $n$ and $t$.
Taking the constant term with respect to $z_0$, we have
\begin{equation}\label{NP-2}
[z_0^{l}]L_{k_1,k_2}(a,b,c)
=\CT_{z^{(1)}}\sum_{\substack{t_1,\dots,t_{k_1}\geq 0\\ |t|=k_1a-l}}
\prod_{i=1}^{k_1}(-z^{(1)}_i)^{t_i-a}
q^{\binom{a}{2}+\binom{t_i}{2}+(1-a)t_i}\qbinom{a+b}{t_i}\times F_{k_1,k_2}(c).
\end{equation}
Since $F_{k_1,k_2}(c)$ only contributes nonnegative powers of the $z^{(1)}_i$,
the summand in \eqref{NP-2} is not homogeneous in $z^{(1)}$ for $l<0$.
It follows that the constant term of each summand in \eqref{NP-2} with respect to $z^{(1)}$ equals zero, and so is the sum.
\end{proof}

\subsection{The positive powers}

As the statement in the beginning of this section, the degree of $z_0$
in $L_{k_1,k_2}(a,b,c)$ should be $k_1a$ at most.
But we find that $[z_0^l]L_{k_1,k_2}(a,b,c)=0$ if $k_1\geq k_2$ and $l>k_2a$ in this subsection.
That is the content of the next lemma.

\begin{lem}\label{lem-L2}
Let $L_{k_1,k_2}(a,b,c)$ be defined as in \eqref{defi-L1}.
If $k_1\geq k_2$ and the integer $l>k_2a$, then
\[
[z_0^l]L_{k_1,k_2}(a,b,c)=0.
\]
\end{lem}
\begin{proof}
It is clear that the lemma holds for $k_2=k_1$ by the definition of $L_{k_1,k_2}(a,b,c)$.
In the following of the proof, we assume $k_1>k_2$.
Changing $t_i\mapsto a-t_i$ for each $i$ in \eqref{NP-2}, we have
\begin{multline}\label{e-lem-L2-1}
[z_0^{l}]L_{k_1,k_2}(a,b,c)
=\sum_{\substack{t_1,\dots,t_{k_1}\leq a\\ |t|=l}}
\prod_{i=1}^{k_1}(-1)^{-t_i}
q^{\binom{a}{2}+\binom{a-t_i}{2}+(1-a)(a-t_i)}\qbinom{a+b}{a-t_i}\\
\times \CT_{z^{(1)}}\frac{\prod_{1\leq i<j\leq k_1}(z_i^{(1)}/z_j^{(1)})_c(qz_j^{(1)}/z_i^{(1)})_c}
{\prod_{i=1}^{k_1}(z^{(1)}_i)^{t_i}\prod_{j=1}^{k_2}\prod_{i=1}^{k_1}(z_i^{(1)}/z_j^{(2)})_c},
\end{multline}
where $|t|=\sum_{i=1}^{k_1}t_i$.
All the constant term in the sum of \eqref{e-lem-L2-1} vanish by Corollary~\ref{cor-1} if $l=|t|=\sum_{i=1}^{k_1}t_i>k_2a$.
Then the lemma follows.
\end{proof}

In the following of this subsection, we show that
$\prod_{l=1}^{k_2}(q^{c-1-b}z_0/z_l^{(2)})_{a+b+1-c}$ is a factor of $L_{k_1,k_2}(a,b,c)$.
The next lemma implies this.
\begin{lem}\label{lem-L3}
If $b+1\leq c\leq a+b+1$ and $k_1\geq k_2$, then
\[
L_{k_1,k_2}(a,b,c)=0
\]
by taking $z_0=z_l^{(2)}q^m$ for $l=1,\dots,k_2$ and $m=1-a,\dots,b+1-c$.
\end{lem}
\begin{proof}
Carrying out the substitution $z_0=z_l^{(2)}q^m$ in $L_{k_1,k_2}(a,b,c)$ yields
\begin{align*}
L:&=L_{k_1,k_2}(a,b,c)\big|_{z_0=z_l^{(2)}q^m}\\
&=\CT_{z^{(1)}}\frac{\prod_{i=1}^{k_1}(q^mz_l^{(2)}/z_i^{(1)})_a(q^{1-m}z_i^{(1)}/z_l^{(2)})_b
\prod_{1\leq i<j\leq k_1}(z_i^{(1)}/z_j^{(1)})_c(qz_j^{(1)}/z_i^{(1)})_c}
{\prod_{j=1}^{k_2}\prod_{i=1}^{k_1}(z_i^{(1)}/z_j^{(2)})_c}\\
&=\CT_{z^{(1)}}\frac{\prod_{i=1}^{k_1}(-z_l^{(2)}/z_i^{(1)})^aq^{ma+\binom{a}{2}}
(q^{1-a-m}z_i^{(1)}/z_l^{(2)})_{a+b}
\prod_{1\leq i<j\leq k_1}(z_i^{(1)}/z_j^{(1)})_c(qz_j^{(1)}/z_i^{(1)})_c}
{\prod_{i=1}^{k_1}(z_i^{(1)}/z_l^{(2)})_c\prod_{\substack{j=1\\j\neq l}}^{k_2}\prod_{i=1}^{k_1}(z_i^{(1)}/z_j^{(2)})_c}.
\end{align*}
Since $1-a-m\leq 1-a-(-a+1)=0$ and $1-a-m+(a+b-1)=b-m\geq b-(b+1-c)=c-1$,
$(z_i^{(1)}/z_l^{(2)})_c$ is in fact a factor of $(q^{1-a-m}z_i^{(1)}/z_l^{(2)})_{a+b}$ for $i=1,\dots,k_1$.
Hence,
\begin{multline}\label{e-lem-L3-1}
L=\CT_{z^{(1)}}\prod_{i=1}^{k_1}(-z_l^{(2)}/z_i^{(1)})^aq^{ma+\binom{a}{2}}
(q^{1-a-m}z_i^{(1)}/z_l^{(2)})_{a+m-1}(q^{c}z_i^{(1)}/z_l^{(2)})_{b-m-c+1}\\
\times
\frac{\prod_{1\leq i<j\leq k_1}(z_i^{(1)}/z_j^{(1)})_c(qz_j^{(1)}/z_i^{(1)})_c}
{\prod_{\substack{j=1\\j\neq l}}^{k_2}\prod_{i=1}^{k_1}(z_i^{(1)}/z_j^{(2)})_c}.
\end{multline}
From $b+1\leq c$, we have $a+m-1+(b-m-c+1)=a+b-c\leq a-1$.
Then, by expanding $(q^{1-a-m}z_i^{(1)}/z_l^{(2)})_{a+m-1}(q^{c}z_i^{(1)}/z_l^{(2)})_{b-m-c+1}$
in the right-hand side of \eqref{e-lem-L3-1} we know that
$L$ is a finite sum of the form
\begin{equation}\label{e-lem-L3-2}
d\cdot (z_l^{(2)})^{\sum_{i=1}^{k_1}t_i}\CT_{z^{(1)}}\frac{\prod_{1\leq i<j\leq k_1}(z_i^{(1)}/z_j^{(1)})_c(qz_j^{(1)}/z_i^{(1)})_c}
{\prod_{i=1}^{k_1}(z_i^{(1)})^{t_i}\prod_{\substack{j=1\\j\neq l}}^{k_2}\prod_{i=1}^{k_1}(z_i^{(1)}/z_j^{(2)})_c},
\end{equation}
where $d\in \mathbb{Q}(q)$ and $t_i\geq 1$ for $i=1,\dots,k_1$.
By Corollary~\ref{cor-2}, every constant term of the form \eqref{e-lem-L3-2} vanishes.
Thus, the lemma follows.
\end{proof}

To show that $L_{k_1,k_2}(a,b,c)$ is of the form \eqref{e-main},
the last step is to prove
\[
[z_0^0]L_{k_1,k_2}(a,b,c)=M_{k_1}(a,b,c).
\]
By the expression for $L_{k_1,k_2}(a,b,c)$ in \eqref{defi-L1},
\begin{equation*}
[z_0^0]L_{k_1,k_2}(a,b,c)=\CT_{z_0,z^{(1)}}\frac{\prod_{i=1}^{k_1}(z_0/z_i^{(1)})_a(qz_i^{(1)}/z_0)_b
\prod_{1\leq i<j\leq k_1}(z_i^{(1)}/z_j^{(1)})_c(qz_j^{(1)}/z_i^{(1)})_c}
{\prod_{j=1}^{k_2}\prod_{i=1}^{k_1}(z_i^{(1)}/z_j^{(2)})_c}.
\end{equation*}
Only those terms that are homogeneous in $z_0$ and the $z_i^{(1)}$ with (total) degree 0 contribute to the constant term. Thus, only the first term (i.e., 1) of the expansion of the denominator
$\prod_{j=1}^{k_2}\prod_{i=1}^{k_1}(z_i^{(1)}/z_j^{(2)})^{-1}_c$
contributes to $[z_0^0]L_{k_1,k_2}(a,b,c)$.
Then, by the $q$-Morris identity \eqref{q-Morris}
\begin{equation*}
[z_0^0]L_{k_1,k_2}(a,b,c)=\CT_{z_0,z^{(1)}}\prod_{i=1}^{k_1}(z_0/z_i^{(1)})_a(qz_i^{(1)}/z_0)_b
\prod_{1\leq i<j\leq k_1}(z_i^{(1)}/z_j^{(1)})_c(qz_j^{(1)}/z_i^{(1)})_c=M_{k_1}(a,b,c).
\end{equation*}

By Lemmas \ref{lem-L1}--\ref{lem-L3} and the fact that $[z_0^0]L_{k_1,k_2}(a,b,c)=M_{k_1}(a,b,c)$,
it is straightforward to obtain that
$L_{k_1,k_2}(a,b,c)$ is of the form \eqref{e-main} for $b+1\leq c\leq a+b+1$ and $k_1\geq k_2$.

\section{The $b=c-1$ case}\label{sec-bc1}

In this section, we are concerned about the $b=c-1$ case of \eqref{e-main}.

Taking $b=c-1$ in \eqref{e-main} we obtain an explicit constant term identity.
\begin{lem}\label{lem-main-1}
For $k_1\geq k_2$,
\begin{multline}\label{e-bc}
L_{k_1,k_2}(a,c-1,c)
=\CT_{z^{(1)}}\frac{\prod_{i=1}^{k_1}(z_0/z_i^{(1)})_a(qz_i^{(1)}/z_0)_{c-1}
\prod_{1\leq i<j\leq k_1}(z_i^{(1)}/z_j^{(1)})_c(qz_j^{(1)}/z_i^{(1)})_c}
{\prod_{j=1}^{k_2}\prod_{i=1}^{k_1}(z_i^{(1)}/z_j^{(2)})_c}\\
=M_{k_1}(a,c-1,c)\prod_{j=1}^{k_2}(z_0/z^{(2)}_j)_{a}.
\end{multline}
\end{lem}
An immediate consequence of Lemma~\ref{lem-main-1} is the next result.
\begin{prop}\label{prop-1}
For $k_1\geq k_2$, let $\lambda$ be a partition such that $\ell(\lambda)\leq k_1-k_2$.
Then
\begin{align}\label{e-h}
&\CT_{z^{(1)}}\frac{h_{\lambda}\big[\frac{1-q^c}{1-q}Z^{(1)}\big]\prod_{i=1}^{k_1}(z_0/z_i^{(1)})_a(qz_i^{(1)}/z_0)_{c-1}
\prod_{1\leq i<j\leq k_1}(z_i^{(1)}/z_j^{(1)})_c(qz_j^{(1)}/z_i^{(1)})_c}
{z_0^{|\lambda|}\prod_{j=1}^{k_2}\prod_{i=1}^{k_1}(z_i^{(1)}/z_j^{(2)})_c}\\
&\quad=(-1)^{|\lambda|}\prod_{i=1}^{\ell(\lambda)}q^{\binom{\lambda_i}{2}}\qbinom{a}{\lambda_i}
\times M_{k_1}(a,c-1,c)\prod_{j=1}^{k_2}(z_0/z^{(2)}_j)_{a}\nonumber \\
&\quad=h_{\lambda}\Big[\frac{q^a-1}{1-q}\Big]\times M_{k_1}(a,c-1,c)\prod_{j=1}^{k_2}(z_0/z^{(2)}_j)_{a}\nonumber,
\end{align}
where $Z^{(1)}=z_1^{(1)}+\cdots+z_{k_1}^{(1)}$.
\end{prop}
\begin{proof}
Denote by $L$ the constant term in \eqref{e-h}.
By the generating function for the complete symmetric functions in \eqref{e-gfcomplete},
\begin{multline*}
L=\CT_{z^{(1)},z_{k_2+1}^{(2)},\dots,z_{k_2+\ell(\lambda)}^{(2)}}
\prod_{j=1}^{\ell(\lambda)}(z_{k_2+j}^{(2)})^{\lambda_j}
\prod_{i=1}^{k_1}(z_0/z_i^{(1)})_a(qz_i^{(1)}/z_0)_{c-1}
\\
\times \frac{\prod_{1\leq i<j\leq k_1}(z_i^{(1)}/z_j^{(1)})_c(qz_j^{(1)}/z_i^{(1)})_c}
{z_0^{|\lambda|}\prod_{j=1}^{k_2+\ell(\lambda)}\prod_{i=1}^{k_1}(z_i^{(1)}/z_j^{(2)})_c}.
\end{multline*}
Since $k_2+\ell(\lambda)\leq k_1$, we can apply Lemma~\ref{lem-main-1} with
$k_2\mapsto k_2+\ell(\lambda)$ and obtain
\begin{align*}
L&=\CT_{z_{k_2+1}^{(2)},\dots,z_{k_2+\ell(\lambda)}^{(2)}}
z_0^{-|\lambda|}\prod_{j=1}^{\ell(\lambda)}(z_{k_2+j}^{(2)})^{\lambda_j}
M_{k_1}(a,c-1,c)\prod_{j=1}^{k_2+\ell(\lambda)}(z_0/z^{(2)}_j)_{a}\\
&=M_{k_1}(a,c-1,c)\prod_{j=1}^{k_2}(z_0/z^{(2)}_j)_{a}
\CT_{z_{k_2+1}^{(2)},\dots,z_{k_2+\ell(\lambda)}^{(2)}}
z_0^{-|\lambda|}\prod_{j=1}^{\ell(\lambda)}(z_{k_2+j}^{(2)})^{\lambda_j}(z_0/z^{(2)}_{k_2+j})_a.
\end{align*}
Using the well-known $q$-binomial theorem \cite[Theorem 3.3]{andrew-qbinomial}
\[
(y)_n=\sum_{t\geq 0}q^{\binom{t}{2}}\qbinom{n}{t}(-y)^t,
\]
we have
\begin{align*}
L&=M_{k_1}(a,c-1,c)\prod_{j=1}^{k_2}(z_0/z^{(2)}_j)_{a}
\CT_{z_{k_2+1}^{(2)},\dots,z_{k_2+\ell(\lambda)}^{(2)}}
z_0^{-|\lambda|}\prod_{j=1}^{\ell(\lambda)}(z_{k_2+j}^{(2)})^{\lambda_j}
\sum_{t_j\geq 0}q^{\binom{t_j}{2}}\qbinom{a}{t_j}(-z_0/z^{(2)}_{k_2+j})^{t_j}\\
&=M_{k_1}(a,c-1,c)\prod_{j=1}^{k_2}(z_0/z^{(2)}_j)_{a}\times
(-1)^{|\lambda|}\prod_{i=1}^{\ell(\lambda)}q^{\binom{\lambda_i}{2}}\qbinom{a}{\lambda_i}.
\end{align*}
The right-most equality of \eqref{e-h} follows from \eqref{e-h-1} below.
\end{proof}
We complete the proof of Proposition~\ref{prop-1} by the next result.
\begin{lem}
For $a$ and $r$ nonnegative integers,
\begin{equation}\label{e-h-1}
h_{r}\Big[\frac{q^a-1}{1-q}\Big]=
(-1)^{r}q^{\binom{r}{2}}\qbinom{a}{r}.
\end{equation}
\end{lem}
\begin{proof}
By \eqref{e-he},
\[
h_{r}\Big[\frac{q^a-1}{1-q}\Big]=
(-1)^re_r\Big[\frac{1-q^a}{1-q}\Big]
=(-1)^re_r(1,q,\dots,q^{a-1})=(-1)^{r}q^{\binom{r}{2}}\qbinom{a}{r}.
\]
Here the last equality holds by \cite[Page 26, Example 3]{Mac95}.
\end{proof}

By \eqref{modi-complete}, we can rewrite \eqref{e-h} as
\begin{multline}\label{e-h-2}
\CT_{z^{(1)}}\frac{g_{\lambda}(z^{(1)};q,q^c)\prod_{i=1}^{k_1}(z_0/z_i^{(1)})_a(qz_i^{(1)}/z_0)_{c-1}
\prod_{1\leq i<j\leq k_1}(z_i^{(1)}/z_j^{(1)})_c(qz_j^{(1)}/z_i^{(1)})_c}
{z_0^{|\lambda|}\prod_{j=1}^{k_2}\prod_{i=1}^{k_1}(z_i^{(1)}/z_j^{(2)})_c}\\
=g_{\lambda}\Big(\Big[\frac{q^a-1}{1-q^c}\Big];q,q^c\Big)\times M_{k_1}(a,c-1,c)\prod_{j=1}^{k_2}(z_0/z^{(2)}_j)_{a}
\end{multline}
for $k_1\geq k_2$ and $\ell(\lambda)\leq k_1-k_2$.
Then, we obtain a few further results.
\begin{prop}\label{thm-2}
For $F=\mathbb{Q}(q,t)$ and partitions $\lambda^{(1)},\dots,\lambda^{(m)}$,
let $f_{\lambda^{(1)}},\dots,f_{\lambda^{(m)}}\in \Lambda_{F}$ be homogeneous symmetric functions
of degrees $|\lambda^{(1)}|,\dots,|\lambda^{(m)}|$ respectively
such that $f_{\lambda^{(i)}}=\sum_{\ell(\mu)\leq \ell(\lambda^{(i)})}d_{\mu}g_{\mu}$ for $i=1,\dots,m$.
Here $d_{\mu}\in F$.
Then, for $k_1\geq k_2$, $\sum_{i=1}^m\ell(\lambda^{(i)})\leq k_1-k_2$ and any alphabets $X_i$ independent of $z^{(1)}$,
\begin{multline}\label{e-thm-2}
\CT_{z^{(1)}}\prod_{i=1}^m f_{\lambda^{(i)}}\big(\big[Z^{(1)}+X_i\big];q,q^c\big)
\frac{\prod_{i=1}^{k_1}(z_0/z_i^{(1)})_a(qz_i^{(1)}/z_0)_{c-1}
\prod_{1\leq i<j\leq k_1}(z_i^{(1)}/z_j^{(1)})_c(qz_j^{(1)}/z_i^{(1)})_c}
{z_0^{\sum_{i=1}^m|\lambda^{(i)}|}\prod_{j=1}^{k_2}\prod_{i=1}^{k_1}(z_i^{(1)}/z_j^{(2)})_c}\\
=\prod_{i=1}^m f_{\lambda^{(i)}}\Big(\Big[\frac{q^a-1}{1-q^c}+\frac{X_i}{z_0}\Big];q,q^c\Big)\times M_{k_1}(a,c-1,c)\prod_{j=1}^{k_2}(z_0/z^{(2)}_j)_{a}.
\end{multline}
\end{prop}
\begin{proof}
Since the $g_{\nu}$ form a basis of $\Lambda_F$ and $f_{\lambda^{(i)}}=\sum_{\ell(\mu)\leq \ell(\lambda^{(i)})}d_{\mu}g_{\mu}$, we can write
$\prod_{i=1}^mf_{\lambda^{(i)}}=\sum d_{\nu}g_{\nu}$,
where the sum ranges over all partitions $\nu$ such that
$\ell(\nu)\leq \sum_{i=1}^m\ell(\lambda^{(i)})\leq k_1-k_2$.
Then, by the linearity it suffices to prove that
\begin{multline}\label{e-thm2-1}
L:=\CT_{z^{(1)}}z_0^{-|\nu|}\prod_{i=1}^{\ell(\nu)}g_{\nu_i}\big[Z^{(1)}+X_i\big]
\prod_{i=1}^{k_1}(z_0/z_i^{(1)})_a(qz_i^{(1)}/z_0)_{c-1}\times F_{k_1,k_2}(c)\\
=\prod_{i=1}^{\ell(\nu)}g_{\nu_i}\Big[\frac{q^a-1}{1-q^c}+\frac{X_i}{z_0}\Big]\times M_{k_1}(a,c-1,c)\prod_{j=1}^{k_2}(z_0/z^{(2)}_j)_{a}
\end{multline}
for a partition $\nu$ such that $\ell(\nu)\leq k_1-k_2$.
Here and in the following of the proof $g_{\nu}=g_{\nu}(q,q^c)$, and $F_{k_1,k_2}(c)$ and
$M_{k_1}(a,b,c)$ are defined in \eqref{defi-F} and \eqref{defi-M} respectively.
By \eqref{e-g}, we can write $L$ as
\begin{align*}
L&=\CT_{z^{(1)}}z_0^{-|\nu|}\prod_{i=1}^{\ell(\nu)}\sum_{t_i=0}^{\nu_i}g_{t_i}[X_i]g_{\nu_i-t_i}[Z^{(1)}]
\prod_{i=1}^{k_1}(z_0/z_i^{(1)})_a(qz_i^{(1)}/z_0)_{c-1}\times F_{k_1,k_2}(c)\\
&=\sum_{\substack{0\leq t_i\leq \nu_i\\i=1,\dots,\ell(\nu)}}
z_0^{-|t|}\prod_{i=1}^{\ell(\nu)}g_{t_i}[X_i]
\CT_{z^{(1)}}z_0^{-|\nu|+|t|}
\prod_{i=1}^{\ell(\nu)}g_{\nu_i-t_i}[Z^{(1)}]
\prod_{i=1}^{k_1}(z_0/z_i^{(1)})_a(qz_i^{(1)}/z_0)_{c-1}\times F_{k_1,k_2}(c),
\end{align*}
where $|t|=\sum_{i=1}^{\ell(\nu)}t_i$.
By Proposition~\ref{prop-1} or its deformation \eqref{e-h-2},
\begin{align*}
L&=\sum_{\substack{0\leq t_i\leq \nu_i\\i=1,\dots,\ell(\nu)}}z_0^{-|t|}
\prod_{i=1}^{\ell(\nu)}g_{t_i}[X_i]
g_{\nu_i-t_i}\Big[\frac{q^a-1}{1-q^c}\Big]\times M_{k_1}(a,c-1,c)\prod_{j=1}^{k_2}(z_0/z^{(2)}_j)_{a}\\
&=\prod_{i=1}^{\ell(\nu)}\sum_{t_i=0}^{\nu_i}
z_0^{-t_i}g_{t_i}[X_i]g_{\nu_i-t_i}\Big[\frac{q^a-1}{1-q^c}\Big]
\times M_{k_1}(a,c-1,c)\prod_{j=1}^{k_2}(z_0/z^{(2)}_j)_{a}.
\end{align*}
Using \eqref{e-homo-sym} and \eqref{e-g} gives
\[
L=\prod_{i=1}^{\ell(\nu)}g_{\nu_i}\Big[\frac{q^a-1}{1-q^c}+\frac{X_i}{z_0}\Big]
\times M_{k_1}(a,c-1,c)\prod_{j=1}^{k_2}(z_0/z^{(2)}_j)_{a}. \qedhere
\]
\end{proof}

Before giving the next result, we introduce a determinant transformation formula between the $p_n$ and the $h_n$.
\begin{lem}\emph{\cite[Page 28]{Mac95}}
For a nonnegative integer $n$,
\begin{equation}\label{e-det-ph}
(-1)^{n-1}p_n=
\left|\begin{array}{ccccc}
h_1 &1 &0 &\cdots &0\\
2h_2 &h_1 &1 &\cdots &0\\
\vdots &\vdots &\vdots &\vdots &\vdots \\
nh_n &h_{n-1} &h_{n-2} &\cdots &h_1
\end{array}\right|.
\end{equation}
\end{lem}

We are concerned about constant term identities for symmetric functions in $\Lambda_{\mathbb{Q}(q,t)}$ in Theorem~\ref{thm-2}. The next result
is a similar result about symmetric functions in $\Lambda_{\mathbb{Q}}$.
\begin{prop}\label{thm-3}
For partitions $\lambda^{(1)},\dots,\lambda^{(m)}$, let $f_{\lambda^{(1)}},\dots,f_{\lambda^{(m)}}\in \Lambda_{\mathbb{Q}}$ be homogeneous symmetric functions of degrees
$|\lambda^{(1)}|,\dots,|\lambda^{(m)}|$ respectively.
Let $X_i$ and $Y_i$ be alphabets independent of $z^{(1)}$ for $i=1,\dots,m$.
Then, for $k_1\geq k_2$ and $\sum_{i=1}^m|\lambda^{(i)}|\leq k_1-k_2$
\begin{multline}\label{e-thm3-1}
\CT_{z^{(1)}}\prod_{i=1}^m f_{\lambda^{(i)}}\big[Y_iZ^{(1)}+X_i\big]
\frac{\prod_{i=1}^{k_1}(z_0/z_i^{(1)})_a(qz_i^{(1)}/z_0)_{c-1}
\prod_{1\leq i<j\leq k_1}(z_i^{(1)}/z_j^{(1)})_c(qz_j^{(1)}/z_i^{(1)})_c}
{z_0^{\sum_{i=1}^m|\lambda^{(i)}|}\prod_{j=1}^{k_2}\prod_{i=1}^{k_1}(z_i^{(1)}/z_j^{(2)})_c}\\
=\prod_{i=1}^m f_{\lambda^{(i)}}\Big[\frac{q^a-1}{1-q^c}Y_i+\frac{X_i}{z_0}\Big]\times M_{k_1}(a,c-1,c)\prod_{j=1}^{k_2}(z_0/z^{(2)}_j)_{a}.
\end{multline}
\end{prop}
\begin{proof}
It suffices to prove that \eqref{e-thm3-1} holds for power sum symmetric functions
by the linearity since the $p_{\lambda}$ form a basis
of $\Lambda_{\mathbb{Q}}$. That is
\begin{multline}\label{e-thm3-2}
L:=\CT_{z^{(1)}}z_0^{-|\lambda|}\prod_{i=1}^{l}p_{\lambda_i}\big[Y_iZ^{(1)}+X_i\big]
\prod_{i=1}^{k_1}(z_0/z_i^{(1)})_a(qz_i^{(1)}/z_0)_{c-1}\times F_{k_1,k_2}(c)\\
=\prod_{i=1}^{l}p_{\lambda_i}\Big[\frac{q^a-1}{1-q^c}Y_i+\frac{X_i}{z_0}\Big]\times M_{k_1}(a,c-1,c)
\prod_{j=1}^{k_2}(z_0/z^{(2)}_j)_{a},
\end{multline}
where $|\lambda|\leq k_1-k_2$, $l=\ell(\lambda)$, the $X_i$ and the $Y_i$ are independent of $z^{(1)}$ for $i=1,\dots,l$, and $F_{k_1,k_2}(c)$ is defined as in \eqref{defi-F}.
By \eqref{asm1} and \eqref{asm3},
\begin{multline*}
L=\CT_{z^{(1)}}z_0^{-|\lambda|}\prod_{i=1}^{l}\Big(p_{\lambda_i}\Big[\frac{1-q}{1-q^c}Y_i\Big]
p_{\lambda_i}\Big[\frac{1-q^c}{1-q}Z^{(1)}\Big]+p_{\lambda_i}[X_i]\Big)\\
\times \prod_{i=1}^{k_1}(z_0/z_i^{(1)})_a(qz_i^{(1)}/z_0)_{c-1} F_{k_1,k_2}(c).
\end{multline*}
Expanding the first product yields
\begin{align*}
L&=\CT_{z^{(1)}}z_0^{-|\lambda|}\sum_{(t_1,\dots,t_l)\in \mathbb{Z}_2^l}
\prod_{i=1}^l\Big(p_{\lambda_i}\Big[\frac{1-q}{1-q^c}Y_i\Big]p_{\lambda_i}\Big[\frac{1-q^c}{1-q}Z^{(1)}\Big]\Big)^{t_i}(p_{\lambda_i}[X_i])^{1-t_i}\\
&\qquad\times \prod_{i=1}^{k_1}(z_0/z_i^{(1)})_a(qz_i^{(1)}/z_0)_{c-1} F_{k_1,k_2}(c)\\
&=\sum_{(t_1,\dots,t_l)\in \mathbb{Z}_2^l}z_0^{-\sum_{i=1}^l\lambda_i(1-t_i)}
\prod_{i=1}^l\Big(p_{\lambda_i}\Big[\frac{1-q}{1-q^c}Y_i\Big]\Big)^{t_i}(p_{\lambda_i}[X_i])^{1-t_i}\\
&\qquad \times
\CT_{z^{(1)}}z_0^{-\sum_{i=1}^l\lambda_it_i}
\prod_{i=1}^l\Big(p_{\lambda_i}\Big[\frac{1-q^c}{1-q}Z^{(1)}\Big]\Big)^{t_i}
\prod_{i=1}^{k_1}(z_0/z_i^{(1)})_a(qz_i^{(1)}/z_0)_{c-1}\times F_{k_1,k_2}(c).
\end{align*}
Write
\begin{equation}\label{e-ph}
\prod_{i=1}^l\big(p_{\lambda_i}\big)^{t_i}
=\sum d_{\nu}h_{\nu}
\end{equation}
using \eqref{e-det-ph}. Since $|\nu|=\sum_{i=1}^l\lambda_it_i\leq |\lambda|\leq k_1-k_2$, we have $\ell(\nu)\leq k_1-k_2$ for each $\nu$ in \eqref{e-ph}.
Thus, we can apply Proposition~\ref{prop-1} and then \eqref{e-ph} to obtain
\begin{align*}
L&=\sum_{(t_1,\dots,t_l)\in \mathbb{Z}_2^l}z_0^{-\sum_{i=1}^l\lambda_i(1-t_i)}
\prod_{i=1}^l\Big(p_{\lambda_i}\Big[\frac{1-q}{1-q^c}Y_i\Big]\Big)^{t_i}(p_{\lambda_i}[X_i])^{1-t_i}\\
&\qquad \times\prod_{i=1}^l\Big(p_{\lambda_i}\Big[\frac{q^a-1}{1-q}\Big]\Big)^{t_i}\cdot
M_{k_1}(a,c-1,c)\prod_{j=1}^{k_2}(z_0/z^{(2)}_j)_{a}\\
&=\sum_{(t_1,\dots,t_l)\in \mathbb{Z}_2^l}\prod_{i=1}^l
\Big(p_{\lambda_i}\Big[\frac{q^a-1}{1-q^c}Y_i\Big]\Big)^{t_i}\Big(p_{\lambda_i}\Big[\frac{X_i}{z_0}\Big]\Big)^{1-t_i}
\cdot M_{k_1}(a,c-1,c)\prod_{j=1}^{k_2}(z_0/z^{(2)}_j)_{a}\\
&=\prod_{i=1}^l\Big(p_{\lambda_i}\Big[\frac{q^a-1}{1-q^c}Y_i\Big]+p_{\lambda_i}\Big[\frac{X_i}{z_0}\Big]\Big)
\cdot M_{k_1}(a,c-1,c)\prod_{j=1}^{k_2}(z_0/z^{(2)}_j)_{a}.
\end{align*}
By \eqref{asm1},
\[
L=\prod_{i=1}^lp_{\lambda_i}\Big[\frac{q^a-1}{1-q^c}Y_i+\frac{X_i}{z_0}\Big]
M_{k_1}(a,c-1,c)\prod_{j=1}^{k_2}(z_0/z^{(2)}_j)_{a}.
\]
Hence, we complete the proof of \eqref{e-thm3-2} and the proposition follows.
\end{proof}

For a positive integer $n$, let $m_1,\dots,m_n$ be nonnegative integers.
For a fixed $s\in \{1,\dots,n\}$, let $\lambda^{(s,j)}$ be partitions for $j=1,\dots,m_s$.
Recall that $z^{(s)}=(z^{(s)}_1,\dots,z^{(s)}_{k_s})$ for $s=1,\dots,n$ and $z=(z^{(1)},z^{(2)},\dots,z^{(n)})$.
By iterating \eqref{e-thm-2} with $z_0=1$, we obtain the next more general result.
\begin{thm}\label{thm-4}
Let $k_1,\dots,k_{n+1}$ and $a_1,\dots,a_n$ be nonnegative integers
such that $k_1\geq k_2\geq \cdots \geq k_{n+1}$ and $a_1=a,a_2=\cdots=a_n=0$.
Let $f_{\lambda^{(s,j)}}\in \Lambda_{\mathbb{Q}(q,t)}$ be homogeneous symmetric functions of degrees
$|\lambda^{(s,j)}|$ such that $f_{\lambda^{(s,j)}}=\sum_{\ell(\mu)\leq \ell(\lambda^{(s,j)})}d_{\mu}g_{\mu}$, where
$d_{\mu}\in \mathbb{Q}(q,t)$.
If $\sum_{j=1}^{m_{s}}\ell(\lambda^{(s,j)})\leq k_{s}-k_{s+1}$ for every $s$
and all the alphabets $X^{(s)}_j$ are independent of $z$, then
\begin{align}
&\CT_{z}
\prod_{s=1}^{n}
\frac{\prod_{i=1}^{k_s}(1/z_i^{(s)})_{a_s}(qz_i^{(s)})_{c-1}
\prod_{1\leq i<j\leq k_s}(z_i^{(s)}/z_j^{(s)})_c(qz_j^{(s)}/z_i^{(s)})_c}
{\prod_{j=1}^{k_{s+1}}\prod_{i=1}^{k_s}(z_i^{(s)}/z_j^{(s+1)})_c}\\
&\quad \times \prod_{s=1}^n\prod_{j=1}^{m_s} f_{\lambda^{(s,j)}}\big(\big[Z^{(s)}+X^{(s)}_j\big];q,q^c\big) \nonumber\\
&\qquad=\prod_{s=1}^n\prod_{j=1}^{m_s} f_{\lambda^{(s,j)}}\Big(\Big[\frac{q^a-1}{1-q^c}+X^{(s)}_j\Big];q,q^c\Big)
\prod_{s=1}^nM_{k_s}(a,c-1,c)
\prod_{j=1}^{k_{n+1}}(1/z^{(n+1)}_j)_{a}.\nonumber
\end{align}
\end{thm}

Similar to Theorem~\ref{thm-4}, we can also iterate \eqref{e-thm3-1} with $z_0=1$ and obtain the next result.
\begin{thm}\label{thm-5}
Let the $k_i$, $m_i$ and $a_i$ be the same as in Theorem~\ref{thm-4}.
Let $f_{\lambda^{(s,j)}}\in \Lambda_{\mathbb{Q}}$ be  homogeneous symmetric functions of degrees
$|\lambda^{(s,j)}|$ such that $\sum_{j=1}^{m_{s}}|\lambda^{(s,j)}|\leq k_s-k_{s+1}$ for every $s$.
If the $X^{(s)}_j$ and the $Y^{(s)}_j$ are alphabets independent of $z$, then
\begin{multline}
\CT_{z}\prod_{s=1}^n\prod_{j=1}^{m_s} f_{\lambda^{(s,j)}}\big[A^{(s,j)}\big]
\prod_{s=1}^n\frac{\prod_{i=1}^{k_s}(1/z_i^{(s)})_{a_s}(qz_i^{(s)})_{c-1}
\prod_{1\leq i<j\leq k_s}(z_i^{(s)}/z_j^{(s)})_c(qz_j^{(s)}/z_i^{(s)})_c}
{\prod_{j=1}^{k_{s+1}}\prod_{i=1}^{k_s}(z_i^{(s)}/z_j^{(s+1)})_c}\\
=\prod_{s=1}^n\prod_{j=1}^{m_s} f_{\lambda^{(s,j)}}\Big[\frac{q^a-1}{1-q^c}Y^{(s)}_j+X^{(s)}_j\Big]
\prod_{s=1}^nM_{k_s}(a,c-1,c)
\prod_{j=1}^{k_{n+1}}(1/z^{(n+1)}_j)_{a},
\end{multline}
where $A^{(s,j)}=(z^{(s)}_1+\cdots+z^{(s)}_{k_s})Y^{(s)}_j+X^{(s)}_j$.
\end{thm}

\section{Equivalent constant term identities}\label{sec-equivalent}

In this section, we give several constant term identities equivalent to those in Section~\ref{sec-bc1}.
All these equivalence are built on the next lemma.
\begin{lem}\label{lem-main-2}
For $k_1\geq k_2$ and $a+b+1\geq c$,
\begin{multline}\label{e-bc2}
\CT_{z^{(1)}}\frac{z_0^{(b+1-c)k_1}\prod_{i=1}^{k_1}(z_0/z_i^{(1)})_a(qz_i^{(1)}/z_0)_b
\prod_{1\leq i<j\leq k_1}(z_i^{(1)}/z_j^{(1)})_c(qz_j^{(1)}/z_i^{(1)})_c}
{\prod_{i=1}^{k_1}(z_i^{(1)})^{b+1-c}\prod_{j=1}^{k_2}\prod_{i=1}^{k_1}(q^{b+1-c}z_i^{(1)}/z_j^{(2)})_c}
\\
=(-1)^{k_1(b+1-c)}q^{k_1\binom{b+2-c}{2}}\prod_{j=1}^{k_2}(z_0/z^{(2)}_j)_{a+b+1-c}\cdot
M_{k_1}(a+b+1-c,c-1,c).
\end{multline}
\end{lem}
Taking $b=c-1$, Lemma~\ref{lem-main-2} reduces to Lemma~\ref{lem-main-1}.
Hence, Lemma~\ref{lem-main-2} is seemingly more general.
But the two lemmas are in fact equivalent.
We will show the equivalence in the following proof of Lemma~\ref{lem-main-2}.
\begin{proof}
We prove the lemma by transforming \eqref{e-bc} to \eqref{e-bc2}.

We can write the left-hand side of \eqref{e-bc} as
\begin{equation*}
\CT_{z^{(1)}}\prod_{i=1}^{k_1}(-z_0/z_i^{(1)})^aq^{\binom{a}{2}}(q^{1-a}z_i^{(1)}/z_0)_{a+c-1}
\frac{\prod_{1\leq i<j\leq k_1}(z_i^{(1)}/z_j^{(1)})_c(qz_j^{(1)}/z_i^{(1)})_c}
{\prod_{j=1}^{k_2}\prod_{i=1}^{k_1}(z_i^{(1)}/z_j^{(2)})_c}.
\end{equation*}
Using the above and taking $a\mapsto a+b+1-c$, \eqref{e-bc} becomes
\begin{multline*}
\CT_{z^{(1)}}\prod_{i=1}^{k_1}(-z_0/z_i^{(1)})^{a+b+1-c}q^{\binom{a+b+1-c}{2}}
(q^{c-a-b}z_i^{(1)}/z_0)_{a+b}
\frac{\prod_{1\leq i<j\leq k_1}(z_i^{(1)}/z_j^{(1)})_c(qz_j^{(1)}/z_i^{(1)})_c}
{\prod_{j=1}^{k_2}\prod_{i=1}^{k_1}(z_i^{(1)}/z_j^{(2)})_c}\\
=\prod_{j=1}^{k_2}(z_0/z^{(2)}_j)_{a+b+1-c}\cdot M_{k_1}(a+b+1-c,c-1,c).
\end{multline*}
Note that the substitution $a\mapsto a+b+1-c$ is valid by $a+b+1\geq c$.
Take $(z_0,z^{(2)}_j)\mapsto (z_0q^{c-b-1},z^{(2)}_jq^{c-b-1})$ in the above. This substitution does not change the constant term. Thus
\begin{multline}\label{e-bc2-1}
\CT_{z^{(1)}}\prod_{i=1}^{k_1}(-q^{c-b-1}z_0/z_i^{(1)})^{a+b+1-c}q^{\binom{a+b+1-c}{2}}
(q^{1-a}z_i^{(1)}/z_0)_{a+b}
\frac{\prod_{1\leq i<j\leq k_1}(z_i^{(1)}/z_j^{(1)})_c(qz_j^{(1)}/z_i^{(1)})_c}
{\prod_{j=1}^{k_2}\prod_{i=1}^{k_1}(q^{b+1-c}z_i^{(1)}/z_j^{(2)})_c}\\
=\prod_{j=1}^{k_2}(z_0/z^{(2)}_j)_{a+b+1-c}\cdot M_{k_1}(a+b+1-c,c-1,c).
\end{multline}
We can rewrite \eqref{e-bc2-1} as
\begin{align}\label{e-bc2-2}
&\CT_{z^{(1)}}\prod_{i=1}^{k_1}(z_0/z_i^{(1)})^{b+1-c}
q^{\binom{a}{2}}(-z_0/z_i^{(1)})^a(q^{1-a}z_i^{(1)}/z_0)_{a+b}
\frac{\prod_{1\leq i<j\leq k_1}(z_i^{(1)}/z_j^{(1)})_c(qz_j^{(1)}/z_i^{(1)})_c}
{\prod_{j=1}^{k_2}\prod_{i=1}^{k_1}(q^{b+1-c}z_i^{(1)}/z_j^{(2)})_c}\\
&\quad=(-1)^{k_1(b+1-c)}q^{k_1\big(\binom{a}{2}+(b+1-c)(a+b+1-c)-\binom{a+b+1-c}{2}\big)}
\nonumber\\
&\qquad\times \prod_{j=1}^{k_2}(z_0/z^{(2)}_j)_{a+b+1-c}\cdot M_{k_1}(a+b+1-c,c-1,c).\nonumber
\end{align}
Substituting
\[
q^{\binom{a}{2}}(-z_0/z_i^{(1)})^a(q^{1-a}z_i^{(1)}/z_0)_{a+b}
=(z_0/z_i^{(1)})_a(qz_i^{(1)}/z_0)_b
\]
and
\[
\binom{a}{2}+(b+1-c)(a+b+1-c)-\binom{a+b+1-c}{2}=\binom{b+2-c}{2}
\]
into \eqref{e-bc2-2} yields \eqref{e-bc2}.
\end{proof}
By Lemma~\ref{lem-main-2} we obtain the next result.
The proof is similar to the proof of Proposition~\ref{prop-1} using Lemma~\ref{lem-main-1}. Hence, we omit the proof.
\begin{prop}\label{prop-1-equiv}
For $a+b+1\geq c$, $k_1\geq k_2$ and a partition $\lambda$ such that $\ell(\lambda)\leq k_1-k_2$,
\begin{multline}\label{e-h-equvi}
\CT_{z^{(1)}}\frac{z_0^{(b+1-c)k_1-|\lambda|}h_{\lambda}\big[\frac{1-q^c}{1-q}Z^{(1)}\big]\prod_{i=1}^{k_1}(z_0/z_i^{(1)})_a(qz_i^{(1)}/z_0)_{b}
\prod_{1\leq i<j\leq k_1}(z_i^{(1)}/z_j^{(1)})_c(qz_j^{(1)}/z_i^{(1)})_c}
{\prod_{i=1}^{k_1}(z_i^{(1)})^{b+1-c}\prod_{j=1}^{k_2}\prod_{i=1}^{k_1}(q^{b+1-c}z_i^{(1)}/z_j^{(2)})_c}\\
\quad=(-1)^{k_1(b+1-c)}q^{k_1\binom{b+2-c}{2}}
h_{\lambda}\Big[\frac{q^{a}-q^{c-b-1}}{1-q}\Big]
\prod_{j=1}^{k_2}(z_0/z^{(2)}_j)_{a+b+1-c}\cdot M_{k_1}(a+b+1-c,c-1,c),
\end{multline}
where $Z^{(1)}=z_1^{(1)}+\cdots+z_{k_1}^{(1)}$.
\end{prop}
The identity \eqref{e-h} in Proposition~\ref{prop-1} is the $b=c-1$ case of \eqref{e-h-equvi}.
However, it is not hard to transform \eqref{e-h} into \eqref{e-h-equvi} along the same line
as the proof of Lemma~\ref{lem-main-2}. Hence, the two identities are equivalent.

By \eqref{modi-complete}, we can rewrite \eqref{e-h-equvi} as
\begin{align}\label{e-lastsec-1}
&\CT_{z^{(1)}}\frac{g_{\lambda}(z^{(1)};q,q^c)\prod_{i=1}^{k_1}(z_0/z_i^{(1)})_a(qz_i^{(1)}/z_0)_{b}
\prod_{1\leq i<j\leq k_1}(z_i^{(1)}/z_j^{(1)})_c(qz_j^{(1)}/z_i^{(1)})_c}
{z_0^{|\lambda|-(b+1-c)k_1}\prod_{i=1}^{k_1}(z_i^{(1)})^{b+1-c}\prod_{j=1}^{k_2}\prod_{i=1}^{k_1}(q^{b+1-c}z_i^{(1)}/z_j^{(2)})_c}\\
&\quad=(-1)^{k_1(b+1-c)}q^{k_1\binom{b+2-c}{2}}
g_{\lambda}\Big(\Big[\frac{q^{a}-q^{c-b-1}}{1-q^c}\Big];q,q^c\Big)\nonumber \\
&\qquad\times \prod_{j=1}^{k_2}(z_0/z^{(2)}_j)_{a+b+1-c}\cdot M_{k_1}(a+b+1-c,c-1,c) \nonumber
\end{align}
for $k_1\geq k_2$, $\ell(\lambda)\leq k_1-k_2$ and $a+b+1\geq c$.
Using \eqref{e-lastsec-1} with $z_0=1$, we obtain the next result equivalent to Proposition~\ref{thm-2}.
The proof is similar to the proof of Proposition~\ref{thm-2}.
\begin{prop}
For $F=\mathbb{Q}(q,t)$ and partitions $\lambda^{(1)},\dots,\lambda^{(m)}$,
let $f_{\lambda^{(1)}},\dots,f_{\lambda^{(m)}}\in \Lambda_{F}$ be homogeneous symmetric functions
of degrees $|\lambda^{(1)}|,\dots,|\lambda^{(m)}|$ respectively
such that $f_{\lambda^{(i)}}=\sum_{\ell(\mu)\leq \ell(\lambda^{(i)})}d_{\mu}g_{\mu}$ for $i=1,\dots,m$.
Here $d_{\mu}\in F$.
Then, for $a+b+1\geq c$, $k_1\geq k_2$, $\sum_{i=1}^m\ell(\lambda^{(i)})\leq k_1-k_2$ and any alphabets $X_i$ independent of $z^{(1)}$
\begin{align}\label{e-lastsec-2}
&\CT_{z^{(1)}}\frac{\prod_{i=1}^mf_{\lambda^{(i)}}
\big(\big[Z^{(1)}+X_i\big];q,q^c)
\prod_{i=1}^{k_1}(1/z_i^{(1)})_a(qz_i^{(1)})_{b}
\prod_{1\leq i<j\leq k_1}(z_i^{(1)}/z_j^{(1)})_c(qz_j^{(1)}/z_i^{(1)})_c}
{\prod_{i=1}^{k_1}(z_i^{(1)})^{b+1-c}\prod_{j=1}^{k_2}\prod_{i=1}^{k_1}(q^{b+1-c}z_i^{(1)}/z_j^{(2)})_c}\\
&\quad=(-1)^{k_1(b+1-c)}q^{k_1\binom{b+2-c}{2}}
\prod_{i=1}^mf_{\lambda^{(i)}}\Big(\Big[\frac{q^{a}-q^{c-b-1}}{1-q^c}+X_i\Big];q,q^c\Big)\nonumber \\
&\qquad\times\prod_{j=1}^{k_2}(1/z^{(2)}_j)_{a+b+1-c}\cdot M_{k_1}(a+b+1-c,c-1,c).\nonumber
\end{align}
\end{prop}

For $n$ a positive integer, let $k_1,\dots,k_{n+1}$, $m_1,\dots,m_n$, $a_1,\dots,a_n$ and $b_1,\dots,b_n$ be nonnegative integers such that $k_1\geq k_2\geq \cdots \geq k_{n+1}$ and $a_1=a,a_2=\cdots=a_n=0$.
For $1\leq s\leq n$, $1\leq j\leq m_s$ and $\lambda^{(s,j)}$ partitions, let
$f_{\lambda^{(s,j)}}\in \Lambda_{F}=\Lambda_{\mathbb{Q}(q,t)}$ be homogeneous symmetric functions of degrees $|\lambda^{(s,j)}|$.
For $s=1,\dots,n$, denote $\sigma_s=b_1+\cdots+b_s$, $\sigma_0=0$, and $Z^{(s)}=z^{(s)}_1+\cdots+z^{(s)}_{k_s}$.
By iterating \eqref{e-lastsec-2} we obtain Theorem~\ref{thm-main-1} below, which is equivalent to Theorem~\ref{thm-4}.
\begin{thm}\label{thm-main-1}
Suppose $f_{\lambda^{(s,j)}}=\sum_{\ell(\mu)\leq \ell(\lambda^{(s,j)})}d_{\mu}g_{\mu}$.
If $a+\sigma_s+s\geq sc$, $\sum_{j=1}^{m_{s}}\ell(\lambda^{(s,j)})\leq k_{s}-k_{s+1}$ for every $s$,
and all the alphabets $X^{(s)}_j$ are independent of $z$, then
\begin{align}\label{e-main-1}
&\CT_{z}
\prod_{s=1}^{n}\frac{\prod_{i=1}^{k_s}(1/z_i^{(s)})_{a_s}(qz_i^{(s)})_{b_s}
\prod_{1\leq i<j\leq k_s}(z_i^{(s)}/z_j^{(s)})_c(qz_j^{(s)}/z_i^{(s)})_c}
{\prod_{i=1}^{k_s}(z_i^{(s)})^{b_s+1-c}\prod_{j=1}^{k_{s+1}}\prod_{i=1}^{k_s}(q^{b_s+1-c}z_i^{(s)}/z_j^{(s+1)})_c}\\
&\qquad \times\prod_{s=1}^n\prod_{j=1}^{m_s} f_{\lambda^{(s,j)}}\big(\big[Z^{(s)}+X^{(s)}_j\big];q,q^c\big)\nonumber\\
&\quad=(-1)^{\sum_{s=1}^nk_s(b_s+1-c)}q^{\sum_{s=1}^nk_s\binom{b_s+2-c}{2}}
\prod_{s=1}^nM_{k_s}\big(a+\sigma_s+s(1-c),c-1,c\big)\nonumber\\
&\qquad\times  \prod_{j=1}^{k_{n+1}}(1/z^{(n+1)}_j)_{a+\sigma_n+n(1-c)}
\prod_{s=1}^n\prod_{j=1}^{m_s} f_{\lambda^{(s,j)}}\Big(\Big[\frac{q^{a+\sigma_{s-1}+(s-1)(1-c)}-q^{c-b_s-1}}
{1-q^c}+X^{(s)}_j\Big];q,q^c\Big).
\nonumber
\end{align}
\end{thm}

Along the same line to obtain Theorem~\ref{thm-5}, we can obtain its equivalent result in the next theorem.
\begin{thm}\label{thm-main-2}
For $s=1,\dots,n$, if $f_{\lambda^{(s,j)}}\in \Lambda_{\mathbb{Q}}$ such that
$\sum_{j=1}^{m_{s}}|\lambda^{(s,j)}|\leq k_s-k_{s+1}$,
all the alphabets $X^{(s)}_j$ and $Y^{(s)}_j$ are independent of $z$, and $a+\sigma_s+s\geq sc$, then
\begin{align}
&\CT_{z}\prod_{s=1}^n\prod_{j=1}^{m_s} f_{\lambda^{(s,j)}}\big[Z^{(s)}Y^{(s)}_{j}+X^{(s)}_j\big]
\prod_{s=1}^{n}\frac{\prod_{i=1}^{k_s}(1/z_i^{(s)})_{a_s}(qz_i^{(s)})_{b_s}
\prod_{1\leq i<j\leq k_s}(z_i^{(s)}/z_j^{(s)})_c(qz_j^{(s)}/z_i^{(s)})_c}
{\prod_{i=1}^{k_s}(z_i^{(s)})^{b_s+1-c}\prod_{j=1}^{k_{s+1}}\prod_{i=1}^{k_s}(q^{b_s+1-c}z_i^{(s)}/z_j^{(s+1)})_c}\\
&\quad=(-1)^{\sum_{s=1}^nk_s(b_s+1-c)}q^{\sum_{s=1}^nk_s\binom{b_s+2-c}{2}}
\prod_{s=1}^n\prod_{j=1}^{m_s} f_{\lambda^{(s,j)}}\Big[\frac{q^{a+\sigma_{s-1}+(s-1)(1-c)}-q^{c-b_s-1}}
{1-q^c}Y^{(s)}_{j}+X^{(s)}_j\Big] \nonumber\\
&\qquad\times \prod_{s=1}^nM_{k_s}\big(a+\sigma_s+s(1-c),c-1,c\big)
\prod_{j=1}^{k_{n+1}}(1/z^{(n+1)}_j)_{a+\sigma_n+n(1-c)}.\nonumber
\end{align}
\end{thm}

\subsection*{Acknowledgements}

This work was supported by the
National Natural Science Foundation of China (No. 12171487).

\end{document}